\documentclass[10pt,reqno]{amsart}
\usepackage{bbm}
\usepackage{mathrsfs}
\usepackage{amsfonts} 
\usepackage[dvipsnames,usenames]{color}
\usepackage{indentfirst}
\usepackage{tikz}
\usepackage{graphicx,latexsym,bm,amsmath,amssymb,verbatim,multicol,lscape}
\newtheorem{thm}{Theorem} [section]
\newtheorem{cor}[thm]{Corollary}
\newtheorem{lem}[thm]{Lemma}
\newtheorem{exm}[thm]{Example}
\newtheorem{prop}[thm]{Proposition}
\theoremstyle{definition}
\newtheorem{defn}[thm]{Definition}
\theoremstyle{remark}
\newtheorem{rem}[thm]{Remark}

\numberwithin{equation}{section}

\begin{document}
\title{On vector parking functions and q-analogue }

\author[Wenkai Yang]{Wenkai Yang}
\address{Center for Applied Mathematics, Tianjin University, Tianjin 300072, P.R. China}
\email{3016210052@tju.edu.cn}

\thanks{This work was supported by National Natural Science Foundation of China (Grant No. 12071344).}

\maketitle

\begin{abstract}
In 2000, it was demonstrated that the set of $x$-parking functions of length $n$, where $x$=($a,b,...,b$) $\in \mathbbm{N}^n$, is equivalent to the set of rooted multicolored forests on [$n$]=\{1,...,$n$\}. 
In 2020, Yue Cai and Catherine H. Yan systematically investigated the properties of rational parking functions. Subsequently, a series of Context-free grammars possessing the requisite property were introduced by William Y.C. Chen and Harold R.L. Yang in 2021.  
In this paper, I discuss generalized parking functions in terms of grammars.
The primary result is to obtain the q-analogue about the number of '1's in certain vector parking functions with the assistance of grammars.
\end{abstract}

\section{Introduction}

The concept of parking functions originated in 1966 through Konheim and Weiss's investigation\cite{ref1} of a prominent computer algorithm. A classical parking function is formally characterized as a positive integer sequence denoted by $a = (a_1,a_2,...,a_n)$, where each $a_{(i)}$, representing the $i$-th non-decreasing order statistic, adheres to the constraint $ a_{(i)} \leq i$ for all $i$ within the range from $1$ to $n$. Moreover, parking functions encompass diverse extensions, in this paper I will only consider $x$-parking functions and vector parking functions.

In 1968, Sch$\ddot{u}$tzenberger\cite{ref17} established a bijection between parking functions of length $n$ and rooted labeled trees on $[n+1]=\{1,...,n+1\}$, both having a cardinality of $(n+1)^{n-1}$. 

Pollak (see Riordan\cite{ref16}), D. E. Knuth\cite{ref3}, D. Foata, and J. Riordan\cite{ref15} demonstrated that with several bijections. 
Employing Shor's recursive methodology, a context-free grammar for the enumeration of rooted trees can be systematically introduced.

\begin{defn}[Formal derivative\cite{ref13}]\label{defn1.1}
	A context-free grammar $G$ is a set of substitution rules defined over an alphabet $A$.  Given a context-free grammar, a formal derivative $D$ can be defined as a differential operator acting on polynomials or Laurent polynomials over $A$. Precisely, $D$ is a linear operator satisfying the following relations:

	(1) $D(u+v)=D(u)+D(v)$ and $D(uv)=D(u)v+uD(v)$;
	
	(2) If $f(x)$ is an analytic function, then $Df(w)= \frac{\partial f(w)}{\partial w} D(w)$;
	
	(3) Following the substitution rule, if $G$ has $v \rightarrow w$ then $D(v)=w$; otherwise, $D(v)=0$, where $v$ is referred to as a constant.
	
	The operator $D$ is regarded as the formal derivative generated by the grammar $G$.
\end{defn} 

To generate rooted trees, a well-known context-free grammar $G$ proposed by Dumont and Ramamonjisoa \cite{ref2} is defined as follows:
\begin{equation}\label{equ1.1}
G:A\rightarrow A^3S, S\rightarrow AS^2.
\end{equation}

The corresponding formal derivative in the grammar is expressed as $D=A^3S \frac{\partial}{\partial A} + AS^2 \frac{\partial}{\partial S}$.  
Furthermore, let $x = (x_1, x_2,..., x_n) \in \mathbbm{N}^n$. 
An $x$-parking function $'a'$ is defined as a sequence of positive integers, where the non-decreasing rearrangement $ a_{(1)} \leq a_{(2)} \leq \cdots \leq a_{(n)}$ satisfies $a_{(i)} \leq x_1 +\cdots +x_i$.
When $x = (1,1,...,1)$, the $x$-parking function degenerates to the classical parking function. In 1999, Pitman and Stanley presented the following result. 

\begin{lem}{\cite{ref5}}\label{lem1.4}
	The quantity of $x$-parking functions, denoted as $P_n(x)$, is defined as:
	\begin{equation}\label{equ1.2}
	P_n(x) =\sum_{(a_1,...,a_n)\in park(n)} x_{a_1}\cdots x_{a_n} \in \mathbbm{N}^n, 
	\end{equation}
	where $park(n)$ is the set of classical parking functions of length $n$. 
Specifically, for the basic $x$-parking function with  $x=(\alpha,\beta,...,\beta)\in \mathbbm{N}^n$, there is  $P_n(x)=\alpha(\alpha+n\beta)^{n-1}.$
\end{lem}

Catherine H.Yan provides a bijection of basic $x$-parking functions to multicolored forests.\cite{ref6} 
On the other hand, the enumerations of rooted trees can be described succinctly with grammar $H$ in \cite{ref4}:
\begin{equation}\label{equ1.3}
H:z\rightarrow zxy, x\rightarrow xyw, y\rightarrow y^3w, w\rightarrow yw^2.
\end{equation}

Let $a$ and $b$ be coprime positive integers assumed throughout the section on vector parking functions. In \cite{ref8}, an $(a, b)$-Dyck path is defined as a lattice path from $(0, 0)$ to $(b, a)$ composed of steps $\{\text{North}, \text{East}\}$ that remains weakly below the diagonal $y = \frac{a}{b}x$. 
An $(a, b)$-parking function of length $'b'$ is characterized by an $(a, b)$-Dyck path with the E-steps labeled from the set $\{0, 1,..., b-1\}$, such that the labels increase in each consecutive step of the E-steps. A combinatorial proof establishing a bijection from $(a, b)-PF_b$ to $[a-1]_0\;^{b-1}$ is presented in \cite{ref8}. Additionally, the authors derive the expression for $PF_{nb}(u)$ in terms of Schur functions.

\begin{lem}{\cite{ref14}}\label{lem1.5}
	For any coprime pair $(a, b)$ with $gcd(a, b) = 1$, the number of $(a, b)$-parking functions of length $b$ is precisely $a^{b-1}$. 
\end{lem}

Several results have been obtained regarding q-analogue of classical parking functions. As an illustration, let $c=(c_1,...,c_n)$ represent a classical parking function, 
$$\sum_{c\in \mathcal{PF}_n}q^{Z(c)}=q(q+n)^{n-1},$$ 
where  $Z(c)$ denotes the count of ones in $c$. The aforementioned result was discovered and proven by Foata and Riordan \cite{ref15}. In 2020, Yue Cai and Catherine H.Yan posed the inquiry of extending these q-analogue results to the rational domain in Section 7 \cite{ref8}, and expressed interest in presenting them in the form of symmetric functions.
This paper concludes with a preliminary statement of relevant conclusions.

 Throughout, diverging from $\mathcal{P}_n(x)$ associated with $x$-parking functions, I designate the set of $u$-parking functions as $\mathcal{PF}_n(u)$ or $\mathcal{PF}_n(\widehat{x})$ with $\widehat{x}_i = x_1 +\cdots +x_i=u_i|_1^n$, and the set of $(a, b)$-parking functions as $(a, b)-\mathcal{PF}_b$. The standard font indicates their cardinality. Furthermore, identical elements $a_i=a_j$ with different codes in the parking function should be treated as different elements in the set, such as the cardinality of the set $\{a_r|1\leq r\leq n \} \sim a = (a_1,...,a_n) $ is $n$.
 
 Let $\textbf{J}_a=(j_1,...,j_{u_n})$ be the $specification$ of $a$, where $j_i=Card \{ a_r\in a| a_r=i\}$. The set $spec(k,b)$ is defined as the set of all the specifications of $(k,kb)$-parking functions.

\section{New results and plan of the paper}

This paper is organized as follows. Section 3 describes a context-free grammar applicable to the basic $x$-parking function. In Section 4, I derive a connection between the vector parking function and the periodic $x$-parking function, as well as a context-free grammar that conforms to them.
Building upon this grammar, Section 5 presents q-analogue for some vector parking functions.

In Section 3, I leverage a bijection between $x$-parking functions and labeled multicolored forests initially introduced in \cite{ref6}. 
Subsequently, to generate rooted labeled trees, I apply a context-free grammar to the alphabet representing the root.
Combining these two ﬁndings, the grammar of parking functions is presented.

Corollary 5.6 in \cite{ref9} suggests that $P_n(x) =k^nP_n(z)$, when the $x$-parking functions with $z=(\frac{a}{b},1,...,1)\in \mathbbm{N}^n$ and $x=b\cdot z$.      
However, $'\frac{a}{b}'$ is not always an integer for any integer $a,b$. Therefore, there is a natural motivation to extend the aforementioned corollary from integers to rational conditions. 

In Section 4, 
I delve into vector parking functions of length $kb$ and height $ka$, where $'a'$ and $'b'$ are coprime, and $k\geq 1$. It is worth noting that generalized Gon$\check{c}$arov polynomials can be used to generalize Corollary 5.6\cite{ref9} from integers to complex numbers.

\begin{thm}\label{thm1.4}
	For any $k,n\in \mathbbm{N}$, $u =(u_0,...,u_{n-1})\in \mathbbm{C} ^n$, and $v=ku$, the relation holds $PF_n(v)=k^nPF_n(u)$.
\end{thm}

Consequently, as a consequence of Theorem \ref{thm1.4} and Lemma \ref{thm1.2}, it can be deduced that $(a,b)-PF_b = a^{b-1}$ through fractional $x$-parking function. Moreover, to demonstrate that the periodic $x$-parking function and vector parking function are equal in the cardinality, when both sides of a Dyck path are not coprime ($k>1$), I need to present an abelian identity with a grammar similar to William Y.C. Chen and Harold R.L. Yang\cite{ref4}.
\begin{thm}\label{thm3.1}
	For $n\geq 1$ and $k\geq 1$,
	
	$$\left( \sum_{i=1}^{k} x_i\right)\cdot \left( n + \sum_{i=1}^{k} x_i\right)^{n-1}=\sum_{i_1+\cdots +i_k=n} \binom{n}{i_1,...,i_k}\prod_{j=1}^{k}x_j\cdot (x_j +i_j)^{i_j-1} $$
\end{thm}

Moving forward, Theorem \ref{thm1.5} is subsequently obtained.
\begin{thm}\label{thm1.5}
	A periodic $x$-parking function consisting of $k$-cycles can be decomposed into $k$ blocks, each conforming to the structure of the basic $x$-parking function in Lemma \ref{thm1.2}. 
	For $x=(1, \frac{a-1}{b},\mathop{...}\limits^{(b-3)},\frac{a-1}{b},\frac{a+b-1}{b},\frac{a-1}{b},\mathop{...}\limits^{(b-3)},\frac{a-1}{b},\frac{a+b-1}{b},...)$, I obtain:
	$$P_{kb}(x)= (a,b)-PF_{kb}=\sum_{J\in spec(k,b)	} \binom{kb}{j_1,...,j_k} \prod_{i=1}^{k} \left(1+j_i\frac{a-1}{b}\right)^{j_i-1}, $$
	where $spec(k,b)=\{J=(j_1, ..., j_k)|j_1\geq b,..., j_1+\cdots+j_t \geq tb, ...,j_1+\cdots + j_k = kb,j_i\geq 0\}$. 
\end{thm}

Finally, in Section 5, the main result of q-analogue is obtained by the construction of Proposition \ref{prop4.7} and Lemma \ref{thm1.6}.\ref{prop4.7}

\begin{thm}\label{thm2.4}
	If $a \equiv 1 \; (mod\; b)$, let $\mathcal{PF}_{db}(u)$ denote the set of $(da,db)$-parking functions with $u=(1+\lfloor i \frac{a}{b}  \rfloor)\vert_{i=0}^{db-1} $. The corresponding q-analogue are given by:

	$$ \sum_{c\in \mathcal{PF}_{db}(u)} q^{Z(c)}=  \sum_{J\in spec(d,b)} \binom{db}{j_1,...,j_d} q\left(q+j_1\cdot \frac{a-1}{b}\right)^{j_1-1} \prod_{t=2}^{d}\left(1+j_t\cdot \frac{a-1}{b}\right)^{j_t-1}.
	$$
	
\end{thm}

\section{A grammar for $x$-parking functions}

In 2000, C.H.Yan introduced a bijection between $x$-parking functions and multicolored labeled forests.\cite{ref16} An unique labeled tree $T$ on $[n] \cup \{ 0\}$ is associated with the labeled forest by appending a root $'0'$ to a forest and then distributing different colored edges to $S_1, S_2,...S_k$. 
Subsequently, I will demonstrate that such trees $T$ on $[n] \cup \{ 0\}$ can be described by a context-free grammar.

Motivated by the four operations listed after Theorem 2.1 in \cite{ref4}, and applying them to $S$, I present Lemma \ref{thm1.1} as follows.

\begin{lem}\label{thm1.1}
	Let $D$ denote the formal derivative with respect to the  grammar $G$ defined in (\ref{equ1.1}) and $T(n,k)$ be the number of labeled trees with the size of $A^{n+k-1}S^n$. 
	Then for $n>0$, I have  
	$$D^n(S)= \sum_{k=0}^{n} T(n+1,k)A^{n+k}S^{n+1}.$$
	Specially, by setting $A=S=1$ above, I obtain 
	$D^n(S)|_{A=S=1}=(n+1)^{n-1}$,
	implying the context-free grammar $G$ of classical parking functions.
	
\end{lem}

Continuing from the ideas presented in the bijection\cite{ref3}, I aim to establish the connections between generalized parking functions and grammars, and to provide q-analogue of them. 
In Lemma \ref{thm1.2}, it will be shown that the total number of $x$-parking functions with the type $a^kb^{n-k}$ is equal to the number of labeled trees with $k$-children adjacent to the root. 

\begin{lem}\label{thm1.2}
	Let $D$ be the formal derivative associated with $H$ (\ref{equ1.3}). Setting $y=1$ in $D^n(z)$ for $n\geq 1$, I have
	$$D^n(z)|_{y=1} = \sum_{k=1}^{n}p_k(n)x^kw^{n-k}z.$$
	Moreover, let $X = (a,b,...,b)\in \mathbbm{N}^n$, then, 
	\begin{equation}\label{equ4.1}
	D^n(z)|_{y=z=1,x=a,w=b} = a(a+bn)^{n-1}.
	\end{equation}
	
	This establishes an exact correspondence between the context-free grammar and the quantitative representation of $X$-parking functions. 
\end{lem}

\begin{proof}[Proof of Lemma \ref{thm1.2}.]
	I will now demonstrate the connection along with a combinatorial interpretation of the coefficients. Recall that I use $F_n$ to label the set of rooted trees on $[n]$ with root $'0'$. Furthermore, let $F_{n,k}$ denote the set of rooted trees on $[n]$ with root $'0'$ and $k$ improper edges. 
	For $T \in F_{n,k}$, I assign a weight to a proper edge by $y$ and represent each improper edge of $T$ by double edges, labeled by $y^2$. \cite{ref4}
	Thus, $'z'$ refer to the root $'0'$, $'x'$ for each child of root $'0'$, and for other all other  vertices it is $'w'$, so that the weight of $T$ is equal to:
	$$\omega_n (T) = zx^{deg_T(n)}y^{n+k}w^{n-deg_T(0)}.$$ 
	
	Where the number of children at the root $0$ is denoted by $deg_T(0)$. Then, $\sum_{T\in F_{n+1}}\omega_n (T) = D^n(z)$.
	
	
	To prove the relation (\ref{equ4.1}), I find that such labeled rooted trees on $[n+1]|_0$ with the weight $x^kw^{n-k}$ mean that the root $'0'$ has $k$ children. Thus, the coefficient is equal to the number of labeled planted forests of $k$-trees on $[n]$.\cite{ref7} 
	$$p_k(n)=\binom{n-1}{k-1}n^{n-k}.$$
	
	Concurrently, the expansion of binomial  $a(a+bn)^{(n-1)} = a \sum_{k=0}^{n-1} \binom{n-1}{k} a^k (nb)^{n-1-k}$. Therefore, the coefficient of $a^kb^{n-k}$ is $\binom{n-1}{k-1}n^{n-k} = p_k(n)$. By ensuring $x=a, w=b$, then $D^n(z)|_{z=y=1,x=a,w=b} = a(a+bn)^{n-1}$ is evident (or let $H_1 : z\rightarrow z(x_1+\cdots+ x_a)y, x_i\rightarrow x_iy(w_1+\cdots+ w_b), y\rightarrow y^3(w_1+\cdots+ w_b), w_j\rightarrow w_jy(w_1+\cdots+ w_b)$, then $x_i=w_j=1$).
	
	Building upon the relation (\ref{equ4.1}), a bijection can be established with multicolored labeled trees where the edge from the root to its children  exhibit $'a'$ distinct colors, while $'b'$ colors are assigned to the edge of other nodes. For the trees with weight $x^kw^{n-k}$ in $D^n(z)$, I substitute it by $a^kb^{n-k}$, since there exists $a^kb^{n-k}$ multicolored labeled trees with the same labeling. 
	By the bijection between multicolored labeled forests and $X$-parking functions with $X = (a,b,...,b)$ \cite{ref6}, completing the proof of Lemma \ref{thm1.2}.  
\end{proof}

In Catherine H. Yan's bijection \cite{ref6}, the number of trees in $i$-th forest $S_i$ is equals to the specification  $j_i$ for $c\in \mathcal{P}_l(X)$ with $X=(a,b,...,b)$. In correlation with Lemma \ref{thm1.2}, the exponents of label $x$ in the weight of $T$ equals $deg_T(0)= j_1+\cdots +j_a$. In other words, this means $deg_T(0)= Card \{ c_k\in c| c_k\leq a \}$, where $c\in \mathcal{P}_l(X)$. On the other hand, $Z(c)=j_1$. Denoted by $q$ when $x$ appears in $S_1$, thus $x$ can choose $1\times'q'$ and $(a-1)\times'1'$. Then I have $\sum_{c\in \mathcal{P}_n(X)} q^{Z(c)}=(q+a-1)(q+a-1+bn)^{n-1}=D^n(z)|_{y=z=1,x=q+(a-1),w=b}$.

\section{A grammar for vector parking functions}

For an $(a, b)$-Dyck path, a simple operation that transforms each $E$-step $(i,j)\rightarrow(i+1,j)$ to the point on the left $(i,j)$. Notably, when substituting  $(i,j)$ with $(i+1,j+1)$, the $(a, b)$-Dyck path is converted into $u$-parking functions. 

\begin{exm}\label{exm4.1}
	As an illustration, the $(4,7)$-parking function $(2,0,3,0,1,2,0)$ is equivalent to a $u$-parking function $(3,1,4,1,2,3,1)$, wherein $u=(1,1,2,2,3,3,4)$ exists. 
\end{exm}

As noted in Lemma \ref{lem1.5}, considering rational $x$-parking functions (depicted by the red line in Figure 1), I find that $(a,b)-PF_b=a^{b-1}=(1+b(\frac{a-1}{b}))^{b-1}=P_b(x)$ when $x=(1,\frac{a-1}{b},...,\frac{a-1}{b})$. This implies that $(a, b)$-parking functions still adhere to the established grammar. To illustrate rational parking functions, I extend the Conclusion 5.6 from \cite{ref9} to the field of rational numbers, as demonstrated in the following proof.

\begin{proof}[Proof of Theorem \ref{thm1.4}.]
	As highlighted in generalized Gon$\check{c}$arov polynomials by substituting the difference operator with a delta operator\cite{ref10}, it introduced a series of conclusions applicable for a fixed sequence $Z=(z_i)_{i\geq0}$ with values in $\mathbbm{K}$, where $\mathbbm{K}$ is a field of characteristic zero. 
	In general, there exists an algebraic property for calculating the generalized Gon$\check{c}$arov polynomials given the basic sequence\cite{ref11}: 
	$$p_m(x) =\sum_{i=0}^m \binom{m}{i} p_{m-i} (z_i) t_i(x;\Delta,Z).$$
	
	Setting $x=0$, I derive the explicit formula that the delta operator does not affect the constant coefficient of the generalized Gon$\check{c}$arov polynomial. Let $R_n=\{(B_1,...,B_k)| B_1\uplus \cdots \uplus B_k= [n]\}$ denote the set of all ordered partitions on the set $[n]$, then for $n\geq 1$, \cite{ref11} 
	\begin{align*}
	t_n(0; -Z) &=\sum_{\rho\in R_n}(-1)^{|\rho|} \prod_{i=0}^{k-1}p_{b_{i+1}} (-z_{s_i}) \\
	&= \sum_{\rho\in R_n}(-1)^{|\rho|} p_{b_1} (-z_0)\cdots p_{b_k}(-z_{s_{k-1}}),
	\end{align*}
	where $b_i=|B_i|$, $|\rho|=k$ and $s_i=\sum_{j=1}^i b_j $.

	\tikz[scale=0.9] {
		
		\begin{tikzpicture}[scale=0.9]
		\draw [green!50] (0,0) grid (7,4);
		\fill [red] (0,0) circle (2pt);
		\fill [gray] (1,0) circle (2pt);
		\fill [gray] (2,1) circle (2pt);
		\fill [gray] (3,1) circle (2pt);
		\fill [gray] (4,2) circle (2pt);
		\fill [gray] (5,2) circle (2pt);
		\fill [gray] (6,3) circle (2pt);
		
		\coordinate (a) at (0,0);\coordinate (b) at (1,0);\coordinate (c) at (2,0);\coordinate (d) at (3,0);\coordinate (e) at (3,1);\coordinate (f) at (4,1);\coordinate (g) at (4,2);\coordinate (h) at (5,2);\coordinate (i) at (6,2);\coordinate (j) at (6,3);\coordinate (k) at (7,3);\coordinate (l) at (7,4);\coordinate [label=below:{(Figure 1) Example 4.1}] (y) at (3.5,0);
		
		\draw[-](a) to [edge label=1] (b);
		\draw[-](b) to [edge label=3] (c);
		\draw[-](c) to [edge label=6] (d);
		\draw[-](d) to (e);
		\draw[-](e) to [edge label=4] (f);
		\draw[-](f) to (g);
		\draw[-](g) to [edge label=0] (h);
		\draw[-](h) to [edge label=5] (i);
		\draw[-](i) to (j);
		\draw[-](j) to [edge label=2] (k);
		\draw[-](k) to (l);
		\draw [dashed,gray](a) to (l);
		\draw [dashed,red](a) to (k);
		\end{tikzpicture}
	}
	~     ~
	\tikz[scale=0.45] {
		
		\begin{tikzpicture}[scale=0.45]
		\draw [green!50] (0,0) grid (14,8);
		\fill [red] (0,0) circle (2pt);
		\fill [gray] (1,0) circle (2pt);
		\fill [gray] (2,1) circle (2pt);
		\fill [gray] (3,1) circle (2pt);
		\fill [gray] (4,2) circle (2pt);
		\fill [gray] (5,2) circle (2pt);
		\fill [gray] (6,3) circle (2pt);
		\fill [red] (7,4) circle (2pt);
		\fill [gray] (8,4) circle (2pt);
		\fill [gray] (9,5) circle (2pt);
		\fill [gray] (10,5) circle (2pt);
		\fill [gray] (11,6) circle (2pt);
		\fill [gray] (12,6) circle (2pt);
		\fill [gray] (13,7) circle (2pt);
		\fill [red] (14,8) circle (2pt);
		\fill [red] (6,18/7) circle (2pt);
		\fill [red] (13,46/7) circle (2pt);
		\coordinate (a) at (0,0);\coordinate (b) at (6,18/7);\coordinate (c) at (7,4);\coordinate (d) at (13,46/7);\coordinate (e) at (14,8);\coordinate [label=below:{(Figure 2) Example 4.3}] (y) at (7,0);

		\draw [dashed,gray](a) to (e);
		\draw [dashed,red](a) to (b);
		\draw [dashed,red](b) to (c);
		\draw [dashed,red](c) to (d);
		\draw [dashed,red](d) to (e);
		
		\end{tikzpicture}
	}

Given the corresponding delta operator $\Delta=D$, the basic sequence $\{p_n(x) = x^n\}_{n\geq 0}$ is associated with $D$. 
For any $ \rho\in R_n$ in the function $t_n(0;D,-Z)$, the polynomial $\prod\limits ^{k-1}_{i=0}p_{b_{i+1}} (-z_{s_i}) =\prod\limits_{i=0}^{k-1} (-z_{s_i})^{b_{i+1}}$ have the same degree $\sum\limits_{i=0}^{k-1} b_{i+1} = n$, then $t_n(0; D, -c\cdot Z)=c^n t_n(0; D, -Z)$. 
Replacing $Z$ with the sequence $u$, I have equation (5) in \cite{ref10} for $u$-parking function:
$$PF_n(u) = t_n(0; D, -u_0, ..., -u_{n-1})=(-1)^nt_n(0; D,u_0,...,u_{n-1}).$$

Therefore, when $u_i$ is multiplied by $c$, the final result of $t_n(0; D, -c\cdot Z)=PF_n(c\cdot u)$ turned into $c^n t_n(0; D, -Z)=c^n PF_n(u)$. This completes the proof.
\end{proof}

	As an immediate inference of Theorem \ref{thm1.4}, I now obtain the following proposition.

\begin{prop}
	An $X$-parking function with $X=(b,a-1,...,a-1)$ of length  $'b'$ can be regarded as $'b'$ multiplicative $f=(1,\frac{a-1}{b},...,\frac{a-1}{b})$. Applying the Lemma \ref{thm1.2} on $\widehat{X}$, 
	it follows that $P_b(X)= PF_b(\widehat{X})=b^b\cdot PF_b(\widehat{f})=D^b(z)|_{z=y=1,x=b,w=a-1}=b(ab)^{b-1}$. 
	
	Given the conclusion of multiplication on rational $X^*$-parking functions with $X^*=f$, I obtain the generalized grammar of $(a, b)$-parking functions where $gcd(a, b) = 1$.
\end{prop}

This section subsequently considers the case of $(ka,kb)$-parking functions with $gcd(ka,kb)=k\in \mathbbm{N}_{>1}$.


\begin{exm}\label{exm4.3}
	In  contrast to Example \ref{exm4.1}, I consider $(4,7)$-parking functions of length $7k(k=2)$. 
	In this way, I admit rational periodic $x$-parking functions (the red dotted line in Figure 2) that satisfy the condition: $x=(1,\overbrace{\tfrac{a-1}{b},...,\tfrac{a-1}{b}}^{b-1},\frac{a+b-1}{b},\overbrace{\tfrac{a-1}{b},...,\tfrac{a-1}{b}}^{b-1})|_{a=4,b=7}$. 

\end{exm}

The combinatorial enumeration of the $x$-parking functions with  $x=(a,b,...,b,d,c,...,c)$ was established as Corollary 2 in 2000 by Catherine H. Yan\cite{ref12}. 
As a further conclusion, I extend $k$ from $2$ to any large positive integer in Theorem \ref{thm1.5} for $x$ in vector form.
Moving forward, the correspondence between the vector parking function and the grammar will be obtained, which can be applied to the q-analogue.

Jiordan articulates the Abel identities in the context of two variables. To extend (5.10) in \cite{ref4} to multiple variables, I utilize the grammar $H$' from \cite{ref4}.

\begin{proof}[Proof of Theorem \ref{thm3.1}.]
	For $i$ from $1$ to $k$, let 
	$$H\text{'}:z_i\rightarrow z_ix_iy, x_i\rightarrow x_iyw, y\rightarrow y^3w, w\rightarrow yw^2.$$
	
	Let $D$ denote the formal derivative associated with $H$'. Utilizing equation (\ref{equ4.1}), $D^n(z_i) |_{z_i=y=w=1} =x_i(x_i+n)^{n-1}$. Since $D(z_1\cdots z_k)=z_1\cdots z_k(x_1+\cdots+ x_k)y$ and $D(x_1+\cdots+ x_k)=(x_1+\cdots+ x_k)yw$, treating $x_1+\cdots+ x_k$ as $X$ and $z_1\cdots z_k$ as $Z$, then $D^n(Z)|_{Z=y=w=1}=X(X+n)^{n-1}$. On the other hands, it is follows from the Leibnitz formula that
	$$D^n(z_1\cdots z_k)|_{Z=y=w=1}=\sum_{i_1+\cdots +i_k=n} \binom{n}{i_1,...,i_k}\prod_{j=1}^{k} D^{i_j}(z_j)|_{z_j=y=w=1} .$$
	
	Combining the results from the previous two parts, this completes the proof.
	
\end{proof}

\begin{cor}\label{cor3}
	Given $x_j=\frac{n}{x-k}$, then
	$$kx^{n-1}= \sum_{i_1+\cdots +i_k=n} \binom{n}{i_1,...,i_k}\prod_{j=1}^{k} (1 +i_j\frac{x-k}{n})^{i_j-1}.$$
\end{cor}

\begin{proof}[Proof of Theorem \ref{thm1.5}.]
	Let me commence the proof with Theorem \ref{thm1.4}: multiply $'x'$ by $'b'$, denoted as $y=b\cdot x=(b,\overbrace{a-1,...,a-1}^{b-1},\underbrace{b+a-1,\overbrace{a-1,...,a-1}^{b-1},...}_{(k-1) blocks})$. Then I decompose the periodic $x$-parking function $f=(s_1,...,s_{kb})$ into $k$ basic $x$-parking functions such that $j$ is the largest integer holding for all $i=0,...,j$:
	$$ Card \{ s_r |s_r\leq b+(b-1+i)(a-1)\} \geq b+i.$$

	\indent Therefore, those $j_1 = b+j $ elements $\{s_r |s_r\leq b+(j_1-1)(a-1)\}$ correspond to $f_1$. And the following ordered elements $s_{(b+j+m)}$ lying between $b+(b+j)(a-1)$ and $2b+(b+j+m-1)(a-1)$ satisfies the condition of basic $x$-parking functions as well (if $s_{(b+j+1)}>2b+(b+j)(a-1)$, then $f_2$ is empty). It holds that there exists $j_2$ as the greatest integer fit to the condition $' Card \{ s_r |b+(b+j)(a-1)\leq s_r\leq 2b+(b+j+i-1)(a-1)\} \geq i'$ for all $0\leq i\leq j_2$. Similarly, $f_2 \sim \{ s_r |b+j_1(a-1)\leq s_r\leq 2b+(j_1+j_2-1)(a-1)\}$. Let's continue the above process and eventually split out $k$ sequences in turn, $f_1,...,f_k$ with lengths $J=(j_1,...j_k)$, where $j_1\geq b,..., j_1+\cdots+j_i \geq ib, ...,j_1+\cdots + j_k = kb$, which is consistent with $J\in spec(k,b)$. Therefore, it follows that
	\begin{equation}\label{equ4.2}
	PF_{kb}(\widehat{y})= b^{kb}\cdot PF_{kb}(\widehat{x})=
	\sum_{J\in spec(k,b)
	} \binom{kb}{j_1,...,j_k}b^k\cdot \prod_{i=1}^{k} (b+j_i(a-1))^{j_i-1}. 
	\end{equation}

	The following proves that $PF_{kb}(\widehat{x}) = (a,b)-PF_{kb}$. Let $ h_j = (1+j \frac{a-1}{b})^{j-1}$, and abbreviate $\binom{mb} {j_1,...,j_m} h_{j_1}\cdots h_{j_m}$ as $\binom{mb}{j_1,...,j_m}_h$. Therefore, I just need to focus on the assignment of $J$. For $\sum_{i=1}^k i\cdot t_i=k$, denote $B(t_1,...,t_k)= \{ \{C_{1,1},...,C_{1,t_1},...,C_{k,t_k}\} |  C_{i,l} = (j_{i,l,1},...,j_{i,l,i}) \; is\; a\; circle,\; j_{i,l,m}\in J, \;1\leq l\leq t_i,\; C_{1,1}\uplus \cdots \uplus C_{k,t_k}=J \}$ the types of circle partition of $J$. Referring to the $Fa\grave{a} di Bruno$ grammar\cite{ref18}, I define $F: \{f_i \rightarrow f_{i+1}g_1,g_i \rightarrow i\cdot g_{i+1} \}$, then
	$D^k(f_0)= \sum\limits_{t=1}^k f_t \sum\limits_{\tiny \begin{array}{c}
		t_1+\cdots+t_k=t\\
		t_1+\cdots+k\cdot t_k=k
		\end{array} } \frac{k!}{t_1!\cdots t_k!\cdot 1^{t_1}\cdots k^{t_k}} g_1^{t_1}\cdots g_k^{t_k}. $
	
	The circle partition argument $Card \; B(t_1,...,t_k)$ is equal to the coefficient of $g_1^{t_1}\cdots g_k^{t_k}$. Now assign values to $J$ such that $j \geq 0$ and $\sigma (C_{i,l}) = j_{i,l,1}+\cdots +j_{i,l,i}= b\cdot i$. 
	Let tiny circle denote an assigned circle $C$ such that $\{S\subsetneq C| \sigma (S) = b\cdot |S| \}= \phi $, then $j\neq b$ for all $j\in C$. 
	By Section 3 in \cite{ref8}, an assigned tiny circle $C_{i,l}$ (appearing $i$ times in rotation) corresponds to a sequence $(j_{i,l,m},j_{i,l,m+1}, ...,j_{i,l,m-1})$ satisfying $spec(i,b)$. Moreover, a circle formed by stitching together $t = t_1+\cdots+t_k$ tiny circles (with $1^{t_1}\cdots k^{t_k}$ times) in a given order can respond to a sequence satisfying $spec(k,b)$. Notice that the operation of tiny circle rotation and spliced circle rotation do not affect each other. Therefore, to cover the permutation, it is necessary to sum over all possible circle divisions of the splice circle. On the other hand, since the value of $j_i$ represents the length of $f_i$, it is necessary to select $j_i$ positions from the $[kb]$.
	$$k!\cdot\sum_{J\in spec(k,b)} \binom{kb}{j_1,...,j_k}_h=
	\sum_{t_1+\cdots+k t_k=k}  \frac{|B(t_1,...,t_k)|\cdot kb!}{b!^{t_1}\cdots (kb)!^{t_k}}\prod_{r=1}^{k} \left[\sum_{j_1+\cdots j_r=rb} \binom{rb}{j_1,...,j_r}_h\right]^{t_r}.
	$$
	
	Substituting Corollary \ref{cor3} with $x=ka,n=kb$, the righthand side simplifies to $k!\sum\limits_{t_1+\cdots+k t_k=k} \frac{kb!}{t_1!\cdots t_k!\cdot b!^{t_1}\cdots (kb)!^{t_k}}\prod \limits_{r=1}^{k} (ra)^{(rb-1)\cdot t_r}$.
	
	By comparing the coefficient of 
	$exp(\sum_{k\geq 1}  \frac{(ka)^{kb-1}}{(kb)!}z^k) $ at $\frac{z^k}{(kb)!} $ introduced by Yue Cai and Catherine H. Yan \cite{ref8} in Corollary 2.9 and Theorem 3.2, I have  $k!\cdot\sum_{J\in spec(k,b)} \binom{kb}{j_1,...,j_k}_h=k!(a,b)-PF_{kb}$. 
	 Verified by (\ref{equ4.2}), $(a,b)-PF_{kb}$ is equivalent to the periodic $x$-parking function $PF_{kb}(\widehat{x})$ as above, this completes the proof of Theorem \ref{thm1.5}.
\end{proof}


Next, I will construct a grammar interpretation. 
There is a bijection between a labeled-planted-forests of $k$-trees with roots $r_i=|T_i|-1$ and an  $(kab, kb)$-parking function.
\begin{prop}\label{prop4.7}
	 ~
	 
	Step 1. I repeat the above partition operation for the periodic $x$-parking function $PF_{kb}(\widehat{x})$ and divide it into $k$ blocks (basic $x$-parking functions): $f_1,...,f_k$ with lengths $j_1,...j_k$. Each block occupies a row without changing the column order of elements $f=(f_1,...,f_k)^T$.
	
	Step 2. To see the classical form of $'f_i'$, I keep the elements of $'f_1'$ the same and subtract $ib+(j_1+ \cdots +j_i)(a-1)$ from each entry in $'f_{i+1}'$. Then each $'f_i'$ attaches an $x$-parking function with $x^*=(b,a-1,...,a-1)$
	
	Step 3. The operation in \cite{ref6} is repeated to turn those $x$-parking functions $'f_i'$ into multicolored trees $T_i$ with a root $r_i=|T_i|-1=j_i$, as in the proof of Lemma \ref{thm1.2}. Notice that every row doesn't affect each other in $'f'$, I add $ib+(j_1+ \cdots +j_i)(a-1)$ to the nodes of tree $T_{i+1}\backslash \{r_{i+1}\}$. Then any vertex $u$ in $T_i\backslash \{r_i\}$ and any vertex $v$ in $T_{i+1}\backslash \{r_{i+1}\}$ has $u\leq ib+(j_1+ \cdots +j_i)(a-1)< v$. So there is only one multicolored labeled forest $'F'$ made up of trees $T_1,...,T_k$. 
	
	In conclusion, I have constructed a mapping from $(a\cdot b,b)-PF_{kb}$ to a multicolored labeled forest. Note that each block $'f_i'$ with $ k_i= Card \{ s_r |(i-1)b+(j_1+ \cdots +j_{i-1})(a-1)\leq s_r\leq ib+(j_1+ \cdots +j_{i-1})(a-1)\} $ corresponds to the tree $T_i$ with $k_i$-children next to the root $r_i$. Let $M(a,b,k)$ be the set of all those forests above.
\end{prop}


By Lemma \ref{thm1.2}, the number of multicolored labeled trees on $[n]$ has been determined as $g(n)=a(a+bn)^{n-1}$. Then using the exponential formula\cite{ref7}, I have $h(n)= \sum_{T|_1^k} g( |T_1|) \cdots g(|T_k|)$ 
where $\{T_1,...,T_k \} $ is a forest partition of $[n]$. I denote by $E_g(t)=\sum_{n\geq 0}g(n)\frac{t^n}{n!}$ and $E_h(t)$ the exponential generating functions of the functions $g$ and $h$, respectively. 

Considering that not all trees that appear in $E_h(t)$ are valid, as requested by $spec(k,b)$. I place a stamp $'i'$ on $T_i$ to distinguish between trees in the forest, and then $E_h(t)=\prod_{i=1}^{k}E_g(t_i)$. In each monomial $t_1^{j_1}\cdots t_k^{j_k}$ of $E_h(t)$, $j_i$ represents the size of $T_i$. However, it is impossible to calculate the coefficients of $E_h(t)$ that satisfy requirement $J\in spec(k,b)$. Nonetheless, utilizing a formal derivative $D$ with respect to grammar $H$ in equation (\ref{equ1.3}), a simple construction as follows.

\begin{lem}
	For $k\geq 1$, let $\delta (\cdot)= D^b(z \; \cdot ) $ as an operator, then
	$\sum\limits_{F\in M(a,b,k)} \omega(F)= \delta^k(1).$
\end{lem}
Whereas, $\delta$ is not a linear operator, which means the Leibnitz formula doesn't take effect. Hence this will make it impossible for the calculation process to proceed smoothly.
To avoid this conundrum, I introduce the step function Paraded $\varepsilon$ and the grammar $K$ in equation (\ref{equ2.2}), aiding us in proving Lemma \ref{thm1.6}.

\begin{lem}\label{thm1.6}
	Define grammar $K$ as follows:
	\begin{equation}\label{equ2.2}
	K:z_i\rightarrow z_ix_iy_it_i, x_i\rightarrow x_iy_iw_it_i, y_i\rightarrow y_i^3w_it_i, w_i\rightarrow y_iw_i^2t_i, t_i\rightarrow 0.
	\end{equation}
	Let $D$ denote the formal derivative with respect to the above grammar $K$. Suppose $z_i=y_i=1,x_i=b,w_i=a-1,t_i^{j_i}=\varepsilon( (j_1-b)+\cdots+(j_i-b)$, where $\varepsilon (x)$ refers to step function Paraded. Then for $gcd(a,b)=1$, I have
	$$D^{kb}(z_1\cdots z_k)|_{\tiny \begin{array}{c}
		z_i=y_i=1,x_i=b,w_i=a-1\\
		t_i^{j_i}=\varepsilon( (j_1-b)+\cdots+(j_i-b) )\end{array} }
	=\sum_{ J\in spec(k,b)	} \binom{kb}{j_1,...,j_k}b^k\cdot \prod_{i=1}^{k} (b+j_i(a-1))^{j_i-1}. $$
\end{lem}

\begin{proof}[Proof of Lemma \ref{thm1.6}.]
	In contrast with Lemma \ref{thm1.2}, I introduce $t_i=|T_i|-1$ as a counter for the tree $T_i$. Derivating for $z_1\cdots z_k$, the weight of the forest $F$ is as follows, 
	$$\omega (F) = \prod_{i=1}^{k} z_ix_i^{deg_T(j_i)}y_i^{j_i+k}w_i^{j_i-deg_T(j_i)}t_i^{j_i}.$$
	\indent Considering the property of the Dirac delta function, let me introduce its integral, 
	\begin{equation*}
	\varepsilon(x)=\left\lbrace 
	\begin{aligned}
	1, x\geq 0; \\
	0, x<0.
	\end{aligned} 
	\right.
	\end{equation*}
	
	Based on the relationship of the rules, I can define the function $h(t_1^{j_1}\cdots t_k^{j_k})=\varepsilon(j_1-b)\cdots \varepsilon( (j_1-b)+\cdots+(j_i-b) ) \cdots $. As will be seen, a context-free grammar for matching forests according to rules is provided. Suppose that $D$ denotes the formal derivative with respect to $K$ in (2.2), then I have $D(z_i)= z_ix_iy_it_i, D(x_i)=x_iy_iw_it_i, D(y_i)= y_i^3w_it_i, D(w_i)= y_iw_i^2t_i, D(t_i)= 0 $. Whereupon, the following conclusions are drawn from the definition of $spec(k,b)$ and equation (\ref{equ4.1}),
	
	\begin{equation*}
	\begin{aligned}
	D^{kb}(z_1&\cdots z_k)|_{\tiny \begin{array}{c}
		z_i=y_i=1,x_i=b,w_i=a-1\\
		t_i^{j_i}=\varepsilon( (j_1-b)+\cdots+(j_i-b) )\end{array} }\\
	=\sum_{J\in spec(k,b)}& \binom{kb}{j_1,...,j_k} \prod_{i=1}^{k} D^{j_i}(z_i)|_{
		z_i=y_i=1,x_i=b,w_i=a-1,t_i=1 }\\
	=\sum_{J\in spec(k,b)}& \binom{kb}{j_1,...,j_k} \prod_{i=1}^{k} b\cdot (b+j_i(a-1))^{j_i-1} .
	\end{aligned}
	\end{equation*}
	Hence Lemma \ref{thm1.6} is proved.
\end{proof}

Note that Lemma \ref{thm1.6} elucidates the enumeration of forests mentioned in equation (\ref{equ4.2}), 
 which is likewise the quantities of periodic $x$-parking functions.
\begin{cor}
	For $gcd(a,b)=1$,
	$$(a,b)-PF_{kb}(= P_{kb}(X)) =D^{kb}(z_1\cdots z_k)|_{\tiny \begin{array}{c}
		z_i=y_i=1,x_i=1,w_i=\frac{a-1}{b}\\
		t_i^{j_i}=\varepsilon( (j_1-b)+\cdots+(j_i-b) )\end{array} }.$$
\end{cor}

	\section{q-analogue for rational parking functions }

In this section I will prove the Theorem \ref{thm2.4}, before which the following lemmas are needed.

\begin{lem}\label{lem5.1}
	For $a \equiv 1 \; (mod\; b)$, the $(a,b)$-parking functions degenerates to the $x$-parking function with $x=(1,\frac{a-1}{b},...,\frac{a-1}{b})$. 
\end{lem}

\begin{proof}
	Assume $k=\frac{a-1}{b}$, then $a=kb+1$ and $gcd(a,b)=1$. For $i$ from $0$ to $b-1$, I have $k\cdot i=\frac{(a-1)i}{b}\leq x_i\leq a\cdot i/b<k\cdot i+1$. 
	Moreover, since $i/b<1$, $x_i=k\cdot i$ is uniquely determined. 
\end{proof}

\begin{lem}\label{lem5.2}
	Given $X=\{ 1,k,...,k\}$ of length $l$, then a q-analogue for $x$-parking function
	$$ \sum_{c\in \mathcal{P}_l(X)} q^{Z(c)}=q(q+lk)^{l-1}.$$
\end{lem}

\begin{proof}[Proof of Lemma \ref{lem5.2}.]
	Following Catherine H.Yan in \cite{ref6}, the parking functions have a  $'+1'$ shift with this paper. According to the discussion on \cite{ref6} at the end of Section 3, the specification $j_1=Z(c)$ for $c\in \mathcal{P}_l(X)$.
	In particular, since $x_1=1$, the forest consist only of $S_1$, so $j_1$ represents the number of trees in the forest. Then refer to Lemma \ref{thm1.2} and Step 3 of Proposition \ref{prop4.7} of this paper to learn about turning a multicolored tree into a forest by adding roots. 
	In contrast to the value of the grammar $'H'$ in  (\ref{equ4.1}), where each $'x=q'$ represents a child of the root, and $'w=k'$ is the remaining nodes with $'k'$ color choices for the edge above it.
\end{proof}

\begin{rem}
	On the other hand, when $'a-1'$ is not divisible by $'b'$, there exists an inequality in the values of the nodes $'w_i'$ and $'w_j'$, such that  the q-analogue cannot be determined. 
\end{rem}

\begin{proof}[Proof of Theorem \ref{thm2.4}.]
	Recall that in a division such as Theorem \ref{thm1.5}. For $(da,db)$-parking function $f=(s_1,...,s_{db})$, let $'j'$ be the largest integer satisfying the condition as follows:
	$$ Card~ \{ s_r |s_r\leq 1+(b-1+i)\frac{a-1}{b}\} \geq b+i,$$   
	holds for all $i=0,...,j$. On the other hand, I have $ Card ~	\{ s_r |s_r \leq 1+(i-1)\frac{a-1}{b} \} \geq i$ for $i=1,...,b$ by Lemma \ref{lem5.1}. 
	Let $k=\frac{a-1}{b}$ and $j_1=b+j$, the sequence $' f_1'$ is an $x$-parking function with $x=\{ 1,k,...,k\}$ of length $j_1$. Denote by $'E'$ the set of codes $'r'$ for which parking function satisfies $s_r\leq 1+(j_1-1)\cdot k$, where $|E|=j_1$. The case $'s_i=1'$ can only exist in $ f_1\sim \{ s_r |r\in E \}=\{s_{(1)},..., s_{(j_1)}\}$, since the parking function has $s_r\geq 2+j_1\cdot k$ with the code $r\notin E$. Furthermore,  $Z(f)= Card ~ \{ s_r\in f |s_r=1 \}= Card ~ \{ r\in E |s_r=1 \}$ since $'1'$ only exist in the first part $'f_1'$. 
	Similarly, according to steps 1-3 of Proposition \ref{prop4.7}, such a $(da,db)$-parking function $'f'$ can be partitioned into $(f_1,... ,f_d)^T$, where $f_t\sim \{ s_r |t+(j_1+\cdots +j_{t-1})k\leq s_r\leq t +(j_1+\cdots +j_t -1)k\}$ for $2\leq t\leq d$, and their lengths $J\in spec(k,b)$.

	To record the parking functions with their $Z(f)$, I use the bijection in Proposition \ref{prop4.7} to map the components $(f_1,... ,f_d)^T$ into a multicolored forest $(S_1,... ,S_d)$ with the bijection, and given a common root $'0'$ make it a multicolored tree. 
	Different colors are used to distinguish the edges between $'0'$ and its children in $S_i$ for $i\leq d$. 	
	In this case, each $'1'$ in $'f_1'$ corresponds to a root $'x_1'$ in $'T(f_1)'$,  which is the q-analogue I need. 
	Following Lemma \ref{lem5.2}, it is sufficient to denote $'x_1'$ by $'q'$ in the grammar of Lemma \ref{thm1.6}, and assign the remaining $x_2,...,x_d$ to $'1'$ and $'w'$ to $'k'$. The proof is obtained.
\end{proof}

%
%

In Corollary 2.11 \cite{ref8}, Yue Cai and Catherine H.Yan provide the number of $(1,b)$-parking functions. When $u=(1,...,1,2,...,2,3,...)$ is of length $nb+b-1$, the resulting expression for $PF_{nb+b-1}(u)$ is $(n+1)^{nb+b-2}$.

\begin{cor}
	For $k=0,~ d=n+1$, Theorem \ref{thm2.4} translates into considering $u=(1,...,1,2,...,2,3,...)$, then
	$$ \sum_{c\in \mathcal{PF}_{db}(u)} q^{Z(c)}=  \sum_{J\in spec(d,b) } \binom{db}{j_1,...,j_d} q^{j_1} .
	$$
	
\end{cor}
\bigskip
\noindent
{\bf Acknowledgment.} I extend my sincere gratitude to William Y.C. Chen for his exceptional supervision and mentorship in the realm of context-free grammar. Additionally, I appreciate the valuable comments and suggestions provided by Catherine H. Yan. Furthermore, I would like to express my thanks to Ming Zeng for his support and guidance in the paper-writing process. This work has received partial support from the National Natural Science Foundation of China with Grant No. 12071344.

\end{document}